\documentclass[12pt]{article}
\usepackage{a4}
\usepackage{amsthm}
\usepackage{amsfonts}
\usepackage{amssymb}
\usepackage{amsmath}
\usepackage{cite}
\usepackage{epsfig}
\usepackage{mathtools}
\usepackage{hyperref}
\usepackage{ifthen}

\newif\iftinymath
\tinymathtrue

\newtheorem{theorem}{Theorem}
\newtheorem{corollary}[theorem]{Corollary}
\newtheorem{lemma}[theorem]{Lemma}
\newtheorem{proposition}[theorem]{Proposition}

\newcommand\BB{{\cal B}}
\newcommand\dd{\,\mbox{d}}
\newcommand\NN{{\mathbb N}}

\newcommand\RR{{\mathbb R}}

\DeclareTextCompositeCommand{\v}{OT1}{l}{l\nobreak\hspace{-.1em}'}
\DeclareTextCompositeCommand{\v}{OT1}{t}{t\nobreak\hspace{-.1em}'\nobreak\hspace{-.15em}}
\newcommand\Wf{\widehat{W}}

\begin{document}
\title{Quasirandom Latin squares\thanks{The work of the first and second authors have received funding from the European Research Council (ERC) under the European Union's Horizon 2020 research and innovation programme (grant agreement No 648509). This publication reflects only its authors' view; the European Research Council Executive Agency is not responsible for any use that may be made of the information it contains. All authors were also supported by the MUNI Award in Science and Humanities of the Grant Agency of Masaryk University.}}

\author{Jacob W. Cooper\thanks{Faculty of Informatics, Masaryk University, Botanick\'a 68A, 602 00 Brno, Czech Republic. E-mail: \tt{\{xcooper,dkral,lamaison,mohr\}@fi.muni.cz}.}\and
	\newcounter{lth} \setcounter{lth}{2}
        Daniel Kr{\'a}\v{l}$^\fnsymbol{lth}$\thanks{Past affiliation: Mathematics Institute, DIMAP and Department of Computer Science, University of Warwick, Coventry CV4 7AL, UK.}\and
        Ander Lamaison$^\fnsymbol{lth}$\and
	Samuel Mohr$^\fnsymbol{lth}$}
\date{} 
\maketitle

\begin{abstract}
We prove a conjecture by Garbe et al.~[arXiv:2010.07854]
by showing that a Latin square is quasirandom if and only
if the density of every $2\times 3$ pattern is $1/720+o(1)$.
This result is the best possible in the sense that
$2\times 3$ cannot be replaced with $2\times 2$ or $1\times n$ for any $n$.
\end{abstract}

\section{Introduction}
\label{sec:intro}

A combinatorial object is said to be quasirandom
if it has properties that a truly random object of the same kind would have almost surely.
The most developed is the theory of \emph{quasirandom graphs},
which can be traced back to the work of R\"odl~\cite{Rod86}, Thomason~\cite{Tho87} and
Chung, Graham and Wilson~\cite{ChuGW89} from the 1980s.
The notion of quasirandom graphs is particularly robust as
several diverse properties of random graphs involving subgraph density,
edge distribution and eigenvalues of the adjacency matrix
are satisfied by a large graph if and only if one of them is.
In particular, if the edge density of a large graph $G$ is $1/2+o(1)$ and
the density of cycles of length four is $1/16+o(1)$,
then the density of all small subgraphs
is close to their expected density in the Erd\H os-R\'enyi random graph with edge density $1/2$.
In other words, the quasirandomness of graphs is captured by densities of two (small) subgraphs.
Results of a similar kind have been obtained for many other types of combinatorial objects,
for example 
groups~\cite{Gow08},
hypergraphs~\cite{ChuG90,Gow06,Gow07,HavT89,KohRS02},
permutations~\cite{Coo04,CooP08,ChaKNPSV20,KraP13},
which also appear in disguise in statistics~\cite{BerD14,EveL21,Hoe48,Yan70},
set systems~\cite{ChuG91s},
subsets of integers~\cite{ChuG92} and
tournaments~\cite{BucLSS21,ChuG91,CorR17,HanKKMPSV19}.
In this paper, we prove a conjecture posed by Garbe, Hancock, Hladk\'y and Sharifzadeh~\cite[Conjecture 12.3]{GarHHS20} and
establish that the same phenomenon holds for Latin squares.
As Garbe et al.\ noted in their paper,
while the setting of Latin squares may initially seem analogous to some of those considered above,
we are in fact faced with the additional challenge of simultaneously managing local and global information,
represented by the relative order of the entries of a Latin square and their global position within the square.
We overcome this difficulty by using their work on the limit theory of Latin squares,
which allows us to approach a compound structure of Latin squares in a compact way.

To state our result precisely, we need to fix some notation.
We use $[n]$ to denote the set $\{1,\ldots,n\}$. 
A \emph{Latin square} of order $n$ is an $n\times n$ matrix such that
each row and each column contains each of the numbers $1,\ldots,n$ exactly once.
A $k\times\ell$ \emph{pattern} is a $k\times \ell$ matrix that
contains each of the numbers $1,\ldots,k\ell$ exactly once.
The \emph{density} of a $k\times\ell$ pattern $A$ in a Latin square $L$ of order $n$,
which we denote by $t(A,L)$,
is the probability that
a uniformly chosen random $k$-tuple $i_1<\cdots<i_k$ of indices between $1$ and $n$ and
a uniformly chosen random $\ell$-tuple $j_1<\cdots<j_\ell$ of indices between $1$ and $n$ satisfy 
$L_{i_aj_b}<L_{i_{a'}j_{b'}}$ if and only if $A_{ab}<A_{a'b'}$ for all $a,a'\in [k]$ and $b,b'\in [\ell]$;
if $k>n$ or $\ell>n$, then we set $t(A,L)$ to be zero.
We say that a sequence $(L_n)_{n\in\NN}$ of Latin squares is \emph{quasirandom}
if it holds for all $k,\ell\in\NN$ that
\begin{equation}
\lim_{n\to\infty} t(A,L_n)=\frac{1}{(k\ell)!}\label{eq:tkl}
\end{equation}
for every $k\times\ell$ pattern $A$.
Our main result (Theorem~\ref{thm:main}) yields that
a sequence $(L_n)_{n\in\NN}$ of Latin squares with orders tending to infinity is quasirandom if and only if
the equality \eqref{eq:tkl} holds for all $2\times 3$ patterns $A$, i.e., when $k=2$ and $\ell=3$,
see Corollary~\ref{cor:main}.

We complement this result by showing that the assumption on the size of the patterns cannot be weakened
such that the equality in \eqref{eq:tkl} holds for all $2\times 2$ patterns $A$, nor that 
\eqref{eq:tkl} holds for all $1\times n$ patterns,
i.e., our result is the best possible in the sense that
$k=2$ and $\ell=3$ are the minimal values of $k$ and $\ell$ such that
the equality in \eqref{eq:tkl} for all $k\times\ell$ patterns imply the quasirandomness.
In particular,
Proposition~\ref{prop:lower} together with the results of Garbe et al.~\cite{GarHHS20}
yields the existence of a sequence $(L_n)_{n\in\NN}$ of Latin squares that is not quasirandom but
\[\lim_{n\to\infty} t(A,L_n)=\frac{1}{\ell!}\]
for every $1\times \ell$ pattern $A$ and 
\[\lim_{n\to\infty} t(A,L_n)=\frac{1}{24}\]
for every $2\times 2$ pattern $A$.

\section{Notation}
\label{sec:notation}

In this section, we fix the notation used throughout the paper, in particular,
the notation concerning combinatorial limits.
We start with some general notation.
The set of positive integers between $1$ and $n$ is denoted by $[n]$.
We often work with the uniform Borel measure on $\RR^n$, which is denoted by $\lambda_n$ or
simply by $\lambda$ when $n$ is clear from the context.

We next define the limit object for Latin squares;
the definitions and results follow the exposition in~\cite{GarHHS20},
where we refer the reader for further details.
In what follows, $\Omega$ can be any separable atomless probability space, however,
we fix $\Omega$ to be $[0,1]$ with the uniform measure for simplicity of our exposition.
Since the order of the points of $[0,1]$ is not important,
we will write $\Omega$ instead of $[0,1]$ to emphasize this fact in what follows.
A \emph{Latinon} is a pair $(W,f)$ such that
$W:\Omega^2\to\BB[0,1]$ is a measurable function,
where $\BB[0,1]$ is the set of Borel probability measures on $[0,1]$ equipped with a natural topology,
$f:\Omega\to [0,1]$ is a measure preserving function, and
\begin{equation}
\int_{\Omega}W(x,y)(S)\dd y=\int_{\Omega}W(y,x)(S)\dd y=\lambda(S)\label{eq:latin}
\end{equation}
for almost all $x\in\Omega$ and all Borel subsets $S\subseteq [0,1]$.
Intuitively speaking, each point $(x,y)\in\Omega^2$ represents an entry,
$W(x,y)$ is the value of the entry, and
$f(x)$ and $f(y)$ determine the position of the row and the column of the entry.

The \emph{density} of a $k\times\ell$ pattern $A$ in a Latinon $(W,f)$
is the probability that $z_{ab}<z_{a'b'}$ if and only if $A_{ab}<A_{a'b'}$ for all $a,a'\in [k]$ and $b,b'\in [\ell]$
where the points $z_{ab}$, $a\in [k]$ and $b\in [\ell]$, are obtained as follows:
sample $k$ and $\ell$ points from $\Omega$,
denote them in a way that $f(x_1)<\cdots<f(x_k)$ and $f(y_1)<\cdots<f(y_\ell)$ (note that
the $f$-images of the points are different with probability one), and
sample a point $z_{ab}$ from $W(x_a,y_b)$ for every $a\in [k]$ and $b\in [\ell]$.
The density of a $k\times\ell$ pattern $A$ in a Latinon $(W,f)$ is denoted by $t(A,W,f)$.
We say that a Latinon $(W,f)$ is \emph{quasirandom}
if $t(A,W,f)=\frac{1}{(k\ell)!}$ for every $k\times\ell$ pattern $A$.

We say that a sequence of Latin squares $(L_n)_{n\in\NN}$ is \emph{convergent}
if the orders of the Latin squares $L_n$ tend to infinity and
the sequence $(t(A,L_n))_{n\in\NN}$ is convergent for every $k\times\ell$ pattern $A$.
A Latinon $(W,f)$ is a \emph{limit} of a convergent sequence of Latin squares $(L_n)_{n\in\NN}$ if
\[\lim_{n\to\infty}t(A,L_n)=t(A,W,f)\]
for every $k\times\ell$ pattern $A$.
The results of Garbe et al.~\cite{GarHHS20} imply that
every convergent sequence of Latin squares has a limit and
every Latinon is a limit of a convergent sequence of Latin squares.
Note that a sequence of Latin squares $(L_n)_{n\in\NN}$ is quasirandom in the sense defined in Section~\ref{sec:intro}
if it is convergent and its limit is a quasirandom Latinon.

\section{Main result}
\label{sec:main}

In this section, we present our main result.
We start with describing a correspondence involving Latinons, which will be useful in the proof.
If $(W,f)$ is a Latinon,
we can define a (measurable) function $\Wf:\Omega^2\times [0,1]\to [0,1]$ as $\Wf(x,y,z)=W(x,y)([0,z])$.
Observe that $\Wf(x,y,z)$ is a non-decreasing function of $z$ when $(x,y)\in\Omega^2$ are fixed and
it holds that
\begin{equation}
\int_{\Omega}\Wf(x,y,z)\dd y=\int_{\Omega}\Wf(y,x,z)\dd y=z\label{eq:Wf}
\end{equation}
for almost all $x\in\Omega$ and all $z\in[0,1]$.
In the other direction,
if $f:\Omega\to [0,1]$ is a measure preserving function and
$\Wf:\Omega^2\times [0,1]\to [0,1]$ is a measurable function such that
\begin{itemize}
\item $\Wf(x,y,z)$ is a non-decreasing function of $z$ when $(x,y)\in\Omega^2$ are fixed,
\item $\Wf(x,y,0)=0$ and $\Wf(x,y,1)=1$ for every $(x,y)\in\Omega^2$, and
\item the equality \eqref{eq:Wf} holds for all $x\in\Omega$ and all $z\in [0,1]$,
\end{itemize}
then $(W,f)$ is a Latinon
where $W:\Omega^2\to\BB[0,1]$ is the unique function satisfying that
$W(x,y)([0,z])=\Wf(x,y,z)$ for all $(x,y)\in\Omega^2$ and $z\in [0,1]$.
Hence, we can identify Latinons $(W,f)$ with pairs $(\Wf,f)$.

In the proof of Theorem~\ref{thm:main},
we will need the following auxiliary lemma,
which is implicit in~\cite[proof of Lemma 2.7 or Lemma 3.3]{LovS11} and
which can be found explicitly stated as~\cite[Lemma 8]{CooKKN18}.

\begin{lemma}
\label{lm:square}
Let $F:\Omega^2\to [0,1]$ be a measurable function such that
\[\int_\Omega F(x,z)F(x',z) \dd z = \xi\]
for almost every $(x,x')\in\Omega^2$.
Then, it holds that
$$\int_\Omega F(x,z)^2 \dd z = \xi$$
for almost every $x\in\Omega$.
\end{lemma}

We are now ready to prove our main result.

\begin{theorem}
\label{thm:main}
Let $(W,f)$ be a Latinon.
If
\[t(A,W,f)=\frac{1}{720}\]
for every $2\times 3$ pattern $A$,
then the Latinon $(W,f)$ is quasirandom.
\end{theorem}

\begin{proof}
Let $\Wf$ be the function associated with $(W,f)$ as defined at the beginning of this section.
We will show for every $z\in [0,1]$ that $\Wf(x,y,z)=z$ for almost all $(x,y)\in\Omega^2$,
which will imply that $(W,f)$ is quasirandom.

Since $(W,f)$ is a Latinon, it holds that
\begin{equation}
\int_{\Omega} \Wf(x,y,z)\dd y=z\label{eq:1}
\end{equation}
for almost all $x\in\Omega$ and every $z\in [0,1]$.
It follows that
\begin{equation}
\int_{\Omega^3} \Wf(x,y_1,z)\Wf(x,y_2,z)\dd x\dd y_1\dd y_2=\int_{\Omega}\left(\int_{\Omega}\Wf(x,y,z)\dd y\right)^2\dd x=z^2,\label{eq:2}
\end{equation}
which yields using the Cauchy-Schwarz Inequality that
\begin{align}
z^4 & = \left(\int_{\Omega^3} \Wf(x,y_1,z)\Wf(x,y_2,z)\dd x\dd y_1\dd y_2\right)^2 \nonumber\\
    & =\left(\int_{\Omega^2}\left(\int_{\Omega} \Wf(x,y_1,z)\Wf(x,y_2,z)\dd x\right)\dd y_1\dd y_2\right)^2 \nonumber\\
    & \le \int_{\Omega^2}\left(\int_{\Omega} \Wf(x,y_1,z)\Wf(x,y_2,z)\dd x\right)^2\dd y_1\dd y_2 \nonumber\\
    & = \int_{\Omega^4}\Wf(x_1,y_1,z)\Wf(x_1,y_2,z)\Wf(x_2,y_1,z)\Wf(x_2,y_2,z)\dd x_1\dd x_2\dd y_1\dd y_2\label{eq:3}
\end{align}
for every $z\in [0,1]$.
In particular, it holds that
{\iftinymath\small\fi
\begin{align}
\frac{1}{5}& =\int_{[0,1]}z^4\dd z\nonumber\\
           & \le\int_{\Omega^4\times [0,1]}\Wf(x_1,y_1,z)\Wf(x_1,y_2,z)\Wf(x_2,y_1,z)\Wf(x_2,y_2,z)\dd x_1\dd x_2\dd y_1\dd y_2\dd z.\label{eq:4}
\end{align}
}

We next show that the assumption of the theorem implies that
{\iftinymath\small\fi
\begin{equation}
\int_{\Omega^4\times [0,1]}\Wf(x_1,y_1,z)\Wf(x_1,y_2,z)\Wf(x_2,y_1,z)\Wf(x_2,y_2,z)\dd x_1\dd x_2\dd y_1\dd y_2\dd z=\frac{1}{5}.\label{eq:5}
\end{equation}
}The integral in \eqref{eq:5} is the probability that
if points $x_1$, $x_2$, $y_1$ and $y_2$ are chosen randomly from $\Omega$,
numbers $z_{ab}$ randomly from $W(x_a,y_b)$ for $a,b\in [2]$, and a number $z$ randomly uniformly from $[0,1]$,
then $z_{ab}\le z$ for all $a,b\in [2]$.
Observe that when $x_1$, $x_2$, $y_1$ and $y_2$ are fixed,
if we choose $y_3$ randomly from $\Omega$ and $z_{13}$ randomly from $W(x_1,y_3)$,
then the number $z_{13}$ is randomly uniformly chosen from $[0,1]$ by \eqref{eq:latin}.
Hence, the integral in \eqref{eq:5} is the probability that
if points $x_1$, $x_2$, $y_1$, $y_2$ and $y_3$ are chosen randomly from $\Omega$,
numbers $z_{ab}$ randomly from $W(x_a,y_b)$ for $a\in [2]$ and $b\in [3]$,
then $z_{ab}\le z_{13}$ for all $a,b\in [2]$.
This probability is equal to the sum of the terms $\alpha_A t(A,W,f)$ over all $2\times 3$ patterns $A$,
where $\alpha_A=1/3$ if the pattern $A$ contains $5$ and $6$ in the same column and $\alpha_A=1/6$ otherwise.
Since the number of $2\times 3$ patterns containing $5$ and $6$ in the same column is $144$,
we obtain that the integral in \eqref{eq:5} is equal to
\[144\times\frac{1}{3}\times\frac{1}{720}+576\times\frac{1}{6}\times\frac{1}{720}=\frac{1}{5}.\]
Hence, the identity \eqref{eq:5} is established.

The identity \eqref{eq:5} implies that equality holds in \eqref{eq:4}, which implies using \eqref{eq:3} that
{\iftinymath\small\fi
\begin{equation}
\int_{\Omega^4}\Wf(x_1,y_1,z)\Wf(x_1,y_2,z)\Wf(x_2,y_1,z)\Wf(x_2,y_2,z)\dd x_1\dd x_2\dd y_1\dd y_2=z^4\label{eq:6}
\end{equation}
}for almost every $z\in [0,1]$ and so the equality holds in \eqref{eq:3} for almost every $z\in [0,1]$.
Since the left side in \eqref{eq:6} is a non-decreasing function of $z$,
it follows that the identity \eqref{eq:6} holds for every $z\in [0,1]$.
In particular, the equality holds in \eqref{eq:3} for every $z\in [0,1]$.
However, the equality in \eqref{eq:3} can hold only if the function $F_z:\Omega^2\to [0,1]$ defined as
\[F_z(y_1,y_2)=\int_{\Omega} \Wf(x,y_1,z)\Wf(x,y_2,z)\dd x\]
is constant almost everywhere.
Hence, we obtain for every $z\in [0,1]$ that
$F_z(y_1,y_2)=z^2$ for almost every $(y_1,y_2)\in\Omega^2$.
Since the probability spaces $\Omega$ and $\Omega^2$ are isomorphic,
we obtain using Lemma~\ref{lm:square} that $F_z(y,y)=z^2$ for almost every $y\in\Omega$.
It follows that
\begin{equation}
\int_{\Omega^2} \Wf(x,y,z)^2\dd x\dd y=\int_{\Omega}F_z(y,y)\dd y=z^2\label{eq:7}
\end{equation}
for every $z\in [0,1]$.
On the other hand, the equality \eqref{eq:1} yields that
\begin{equation}
\int_{\Omega^2} \Wf(x,y,z)\dd x\dd y=z\label{eq:8}
\end{equation}
for every $z\in [0,1]$.
It follows then, by \eqref{eq:7} and \eqref{eq:8}, that
\[\left(\int_{\Omega^2} \Wf(x,y,z)\dd x\dd y\right)^2=\int_{\Omega^2} \Wf(x,y,z)^2\dd x\dd y,\]
and hence we obtain for every $z\in [0,1]$ that $\Wf(x,y,z)=z$ for almost all $(x,y)\in\Omega^2$.
Consequently, the following holds for almost every $(x,y)\in\Omega^2$: $\Wf(x,y,z)=z$ for every rational $z\in [0,1]$;
here, we use that the set of rationals is countable.
Since the function $\Wf(x,y,z)$ is a non-decreasing function of $z$ when $(x,y)\in\Omega^2$ is fixed,
we obtain for almost every $(x,y)\in\Omega^2$ that $\Wf(x,y,z)=z$ for every $z\in [0,1]$.
In particular, it holds that $W(x,y)$ is the uniform measure on $[0,1]$ for almost every $(x,y)\in\Omega^2$.
This implies that $t(A,W,f)=1/(k\ell)!$ for every $k\times\ell$ pattern $A$ and
we conclude that the Latinon $(W,f)$ is quasirandom.
\end{proof}

We obtain the following corollary.

\begin{corollary}
\label{cor:main}
Let $(L_n)_{n\in\NN}$ be a sequence of Latin squares with their orders tending to infinity such that
\[\lim_{n\to\infty}t(A,L_n)=\frac{1}{720}\]
for every $2\times 3$ pattern $A$.
Then it holds that
\[\lim_{n\to\infty}t(A,L_n)=\frac{1}{(k\ell)!}\]
for every $k\times\ell$ pattern $A$, i.e., $(L_n)_{n\in\NN}$ is quasirandom.
\end{corollary}

\begin{proof}
Fix a sequence $(L_n)_{n\in\NN}$ of Latin squares satisfying the assumption of the corollary, and
suppose that
the limit of $t(A_0,L_n)$ does not exist or is different from $\frac{1}{(k_0\ell_0)!}$ for a $k_0\times\ell_0$ pattern $A_0$.
By considering a subsequence,
we may in fact assume that the limit of $t(A_0,L_n)$ exists and is different from $\frac{1}{(k_0\ell_0)!}$.
Since every sequence of Latin squares has a convergent subsequence~\cite{GarHHS20},
we obtain a convergent sequence $(L'_n)_{n\in\NN}$ of Latin squares with their orders tending to infinity such that
\[\lim_{n\to\infty}t(A,L'_n)=\frac{1}{720}\]
for every $2\times 3$ pattern $A$ and
\[\lim_{n\to\infty}t(A_0,L'_n)\not=\frac{1}{(k_0\ell_0)!}.\]
Let $(W,f)$ be the limit Latinon for the sequence $(L'_n)_{n\in\NN}$.
By Theorem~\ref{thm:main}, the Latinon $(W,f)$ is quasirandom,
which yields that
\[t(A_0,W,f)=\frac{1}{(k_0\ell_0)!}.\]
This contradicts that the limit of $t(A_0,L_n)$ does not exist or is different from $\frac{1}{(k_0\ell_0)!}$, and
we conclude that the sequence $(L_n)_{n\in\NN}$ is quasirandom.
\end{proof}

Finally,
the following corollary confirms the conjecture by Garbe et al.~\cite[Conjecture 13.2]{GarHHS20} as stated in their paper.
The proof of the corollary follows by inspecting the proof of Theorem~\ref{thm:main},
in particular, the conclusion on the values of $W$ at the very end of the proof.

\begin{corollary}
\label{cor:conj}
Let $(W,f)$ be a Latinon.
If $t(A,W,f)=1/720$ for every $2\times 3$ pattern $A$,
then $W(x,y)$ is the uniform measure on $[0,1]$ for almost every $(x,y)\in\Omega^2$.
\end{corollary}

\section{Lower bound}
\label{sec:lower}

In this section, we establish that the assumption in Theorem~\ref{thm:main} involving $2\times 3$ patterns
cannot be replaced with $1\times n$ patterns or $2\times 2$ patterns.

\begin{proposition}
\label{prop:lower}
There exists a Latinon $(W,f)$ that is not quasirandom and
that satisfies for all $n\in\NN$ that
\[t(A,W,f)=\frac{1}{n!}\]
for every $1\times n$ pattern $A$ and
\[t(A,W,f)=\frac{1}{24}\]
for every $2\times 2$ pattern $A$.
\end{proposition}

\begin{proof}
Consider a Latinon $(W,f)$ defined as follows.
First, the function $f:\Omega\to [0,1]$ is defined as
\[f(x)=\begin{cases}
       2x & \mbox{if $x\le 1/2$, and}\\
       2x-1 & \mbox{otherwise.}
       \end{cases}\]
Clearly, the function $f$ is measure preserving.
The function $W:\Omega^2\to\BB[0,1]$ is defined as follows:
\[W(x,y)(S)=\begin{cases}
            2\lambda\left(S\cap [0,1/2]\right) & \mbox{if $(x,y)\in [0,1/2]^2\cup (1/2,1]^2$, and} \\
            2\lambda\left(S\cap [1/2,1]\right) & \mbox{otherwise,}
            \end{cases}\]
for every $(x,y)\in\Omega^2$ and every Borel subset $S\subseteq [0,1]$.	    

We first verify that $t(A,W,f)=\frac{1}{n!}$ for every $1\times n$ pattern $A$.
Fix $n\in\NN$.
Sample a point $x$ randomly from $\Omega$ and
$n$ additional points $y_i$ randomly from $\Omega$ and index them in such a way that $f(y_1)<\cdots<f(y_n)$.
The choice of the function $f$ implies that $y_i\in [0,1/2]$ with probability $1/2$ for every $i\in [n]$ (note that
this probability does not change when conditioned on that $f(y_i)$ is the $i$-th large $f$-image of the sampled points).
Next sample the point $z_i$ from $W(x,y_i)$ for $i\in [n]$.
If $(x,y_i)\in [0,1/2]^2\cup (1/2,1]^2$, then $z_i$ is a uniformly chosen point from $[0,1/2]$, and
it is a uniformly chosen point from $(1/2,1]$ otherwise.
It follows that $z_i$ is a uniformly chosen point from $[0,1]$ independent of the choices of other points,
i.e., $z_1,\ldots,z_n$ are $n$ independent uniformly random points from $[0,1]$.
It follows that $t(A,W,f)=1/n!$ for every $1\times n$ pattern $A$.

We next analyze the density of $2\times 2$ patterns.
Sample two pairs of points from $\Omega$,
index them in such a way that $f(x_1)<f(x_2)$ and $f(y_1)<f(y_2)$, and
then sample points $z_{ij}$ from $W(x_i,y_j)$ for $i,j\in [2]$.
Let $X=[0,1/2]^2\cup (1/2,1]^2$ and $Y=[0,1]^2\setminus X$.
Consider the $2\times 2$ matrix $B$ such that the entry in the $i$-th row and
the $j$-th column is $X$ if $(x_i,y_j)\in X$ and $Y$ otherwise.
An analysis analogous to that in the previous paragraph yields that
$B$ is equal to each of the following eight matrices with probability $1/8$:
\[\begin{pmatrix} X & X \\ X & X \end{pmatrix}
  \begin{pmatrix} Y & Y \\ Y & Y \end{pmatrix}
  \begin{pmatrix} X & X \\ Y & Y \end{pmatrix}
  \begin{pmatrix} Y & Y \\ X & X \end{pmatrix}
  \begin{pmatrix} X & Y \\ X & Y \end{pmatrix}
  \begin{pmatrix} Y & X \\ Y & X \end{pmatrix}
  \begin{pmatrix} X & Y \\ Y & X \end{pmatrix}
  \begin{pmatrix} Y & X \\ X & Y \end{pmatrix}\]
Fix a $2\times 2$ pattern $A$.
For each of the first two matrices above, the probability that $z_{ab}<z_{a'b'}$ if and only if $A_{ab}<A_{a'b'}$ for all $a,b,a',b'\in [2]$,
is equal to $1/24$ conditioned on $B$ being that matrix.
Among the remaining six matrices,
there is exactly one such that
the probability is equal to $1/4$ conditioned on $B$ being this matrix (it is the matrix obtained from $A$
by replacing the entries equal to $1$ and $2$ with $X$ and the entries equal to $3$ and $4$ with $Y$), and
the probability is equal to $0$ conditioned on $B$ being one of the remaining five matrices.
We conclude that
\[t(A,W,f)=\frac{1}{8}\left(\frac{2}{24}+\frac{1}{4}\right)=\frac{1}{24}\,.\]

Since Corollary~\ref{cor:conj} implies that the Latinon is not quasirandom,
the proof is now concluded.
For completeness,
we remark that for example the density of the pattern $2\times 3$ pattern
\[A=\begin{pmatrix} 1 & 2 & 3 \\ 4 & 5 & 6 \end{pmatrix}\]
in the Latinon $(W,f)$ is $\frac{11}{16\cdot 360}\not=\frac{1}{720}$.
\end{proof}

We present another example showing that the assumption in Theorem~\ref{thm:main} cannot be replaced solely with $2\times 2$ patterns.

\begin{proposition}
\label{prop:lower2}
Let $X=\left([0,1/4]\cup (3/4,1]\right)^2\cup (1/4,3/4]^2$ and $Y=[0,1]^2\setminus X$, and
define $W:\Omega^2\to\BB[0,1]$ as
\[W(x,y)(S)=\begin{cases}
            2\lambda\left(S\cap [0,1/2]\right) & \mbox{if $(x,y)\in X$, and} \\
            2\lambda\left(S\cap [1/2,1]\right) & \mbox{otherwise,}
            \end{cases}\]
for a Borel set $S\subseteq [0,1]$, and $f:\Omega\to [0,1]$ as the identity, i.e., $f(x)=x$.
The Latinon $(W,f)$ is not quasirandom
but it holds that $t(A,W,f)=1/24$ for every $2\times 2$ pattern $A$.
\end{proposition}

\begin{figure}
\begin{center}
\epsfbox{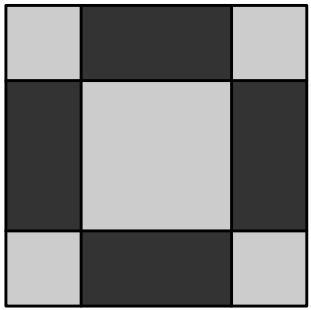}
\end{center}
\caption{Illustration of the structure of the Latinon $(W,f)$ from the statement of Proposition~\ref{prop:lower2}.
         The areas with light gray correspond to those where the support of the measure $W(x,y)$ is $[0,1/2]$ and
	 those with dark gray correspond to those where the support of the measure $W(x,y)$ is $[1/2,1]$.}
\label{fig:lower2}
\end{figure}

\begin{proof}
Corollary~\ref{cor:conj} implies that the Latinon $(W,f)$ is not quasirandom and
so we focus on computing the density of $2\times 2$ patterns in the Latinon $(W,f)$.
Fix a $2\times 2$ pattern $A$.
Sample two pairs of points from $\Omega$,
index them in such a way that $f(x_1)<f(x_2)$ and $f(y_1)<f(y_2)$, and
then sample points $z_{ij}$ from $W(x_i,y_j)$ for $i,j\in [2]$.
Let $P=[0,1/4]\cup (3/4,1]$ and $Q=(1/4,3/4]$, and observe that
the probability that both $x_1$ and $x_2$ belong to $P$ is $1/4$,
the probability that they both belong to $Q$ is also $1/4$,
the probability that $x_1$ belongs to $P$ and $x_2$ to $Q$ is also $1/4$, and
the probability that $x_1$ belongs to $Q$ and $x_2$ to $P$ is $1/4$, too.
Let $B$ now be the $2\times 2$ matrix $B$ such that the entry in the $i$-th row and
the $j$-th column is $X$ if $(x_i,y_j)\in X$ and $Y$ otherwise.
Since $X=P^2\cup Q^2$, it follows that $B$ is equal to each of the following eight matrices with probability $1/8$:
\[\begin{pmatrix} X & X \\ X & X \end{pmatrix}
  \begin{pmatrix} Y & Y \\ Y & Y \end{pmatrix}
  \begin{pmatrix} X & X \\ Y & Y \end{pmatrix}
  \begin{pmatrix} Y & Y \\ X & X \end{pmatrix}
  \begin{pmatrix} X & Y \\ X & Y \end{pmatrix}
  \begin{pmatrix} Y & X \\ Y & X \end{pmatrix}
  \begin{pmatrix} X & Y \\ Y & X \end{pmatrix}
  \begin{pmatrix} Y & X \\ X & Y \end{pmatrix}\]
As in the proof of Proposition~\ref{prop:lower},
we observe that
the probability that $z_{ab}<z_{a'b'}$ if and only if $A_{ab}<A_{a'b'}$ for all $a,b,a',b'\in [2]$,
is equal to $1/24$ conditioned on $B$ being one of the first two matrices, and
among the remaining six matrices,
there is exactly one such that
the probability is equal to $1/4$ conditioned on $B$ being this matrix, and
the probability is equal to $0$ conditioned on $B$ being one of the remaining five matrices.
We conclude that
\[t(W,A,f)=\frac{1}{8}\left(\frac{2}{24}+\frac{1}{4}\right)=\frac{1}{24}\,.\]
The proof of the proposition is now concluded.
\end{proof}

\section{Concluding remarks}

Instead of the notion of patterns considered in this paper,
it is possible to consider a more general notion of a pattern defined as follows:
a $k\times\ell$ \emph{generalized pattern} is a $k\times \ell$ matrix such that
$m$ of its entries are equal to $\star$ and the remaining entries contains each of the numbers $1,\ldots,k\ell-m$ exactly once.
For example, the following is a generalized $2\times 3$ pattern:
\[\begin{pmatrix}
  2 & 3 & \star \\
  \star & 4 & 1 \\
  \end{pmatrix}.\]
The informal meaning of entries equal to $\star$ is that their relative order is not restricted.
Formally, the \emph{density} of a $k\times\ell$ generalized pattern $A$ in a Latin square $L$ of order $n$,
which is denoted by $t(A,L)$,
is the probability that
a uniformly chosen random $k$-tuple $i_1<\cdots<i_k$ of indices between $1$ and $n$ and
a uniformly chosen random $\ell$-tuple $j_1<\cdots<j_\ell$ of indices between $1$ and $n$
satisfy for all $a,a'\in [k]$ and $b,b'\in [\ell]$ that
$L_{i_aj_b}<L_{i_{a'}j_{b'}}$ whenever $A_{ab}\not=\star$, $A_{a'b'}\not=\star$ and $A_{ab}<A_{a'b'}$;
if $k>n$ or $\ell>n$, then we again set $t(A,L)$ to be zero.
We will assume in what follows that a generalized pattern does not contain a row or a column with all entries equal to $\star$
since the removal of such a row or column does not asymptotically change the density of the pattern.

A careful inspection of the proof of Theorem~\ref{thm:main} yields that
the proof only requires the density of each of the following $72$ generalized patterns
\[\begin{pmatrix}
  \le 4 & \le 4 & 5 \\
  \le 4 & \le 4 & \star
  \end{pmatrix}
  \quad
  \begin{pmatrix}
  \le 4 & 5 & \le 4 \\
  \le 4 & \star & \le 4
  \end{pmatrix}
  \quad
  \begin{pmatrix}
  5 & \le 4 & \le 4 \\
  \star & \le 4 & \le 4
  \end{pmatrix}\]
to be equal to $1/120$ where $\le 4$ denotes numbers from $1$ to $4$ in an arbitrary order (note that
it would actually be sufficient to require only a partial order between the non-$\star$ entries of a generalized pattern
for the purpose of the proof of Theorem~\ref{thm:main}).
To complement this,
we claim that
any set of generalized patterns with at most four entries different from $\star$ does not have this property,
i.e., our result is the best possible also in the stronger setting of generalized patterns.
Indeed, consider a sequence of Latin squares converging to the Latinon constructed in the proof of Proposition~\ref{prop:lower}.
We say that a generalized pattern $A$ is \emph{eliminable}
if its entries different from $\star$ can be ordered in such a way that the following is satisfied:
each entry different from $\star$
is either the first element in the ordering among all entries in the same row, or
it is the first among all entries in the same column.
For example, the pattern
\[\begin{pmatrix}
  2 & 3 & \star \\
  \star & 4 & 1 \\
  \end{pmatrix}.\]
is eliminable as witnessed by ordering its entries different from $\star$ as, e.g., $2,3,4,1$;
here, for example, the entry $3$ is the first entry in the ordering among all entries in the same column.
Next observe that the density of any eliminable pattern in any sequence of Latin squares converging
to the Latinon from the proof of Proposition~\ref{prop:lower}
converges to $1/m!$ where $m$ is the number of entries different from $\star$.
Since the only non-eliminable patterns with at most four entries different from $\star$ (and no row or column formed by $\star$ only)
are $2\times 2$ patterns, each of which has the limit density equal to $1/24$ by Proposition~\ref{prop:lower}, the claim follows.

\section*{Acknowledgements}

The second author would like to thank Tomasz \L{}uczak for numerous discussions
on the combinatorial limit approach to extremal problems on Latin squares,
including their quasirandomness, during their visit to the Institut Mittag-Leffler in Djursholm, Sweden, in the spring 2014;
he is particularly indebted for gaining many valuable insights during the discussions.

\bibliographystyle{bibstyle}
\bibliography{qlatin}

\end{document}